\def\mode{1}	
\if 0\mode
\documentclass[journal, twoside, web]{ieeecolor}
\usepackage{lcsys}
\else
\documentclass{IEEEtran}
\IEEEoverridecommandlockouts
\fi

\usepackage[english]{babel}
\usepackage{microtype}
\usepackage{setspace}
\usepackage{blindtext}
\usepackage{siunitx}
\usepackage{titlecaps}
\Addlcwords{or if of and an a to via -off off for in the vs} 
\usepackage{comment}

\usepackage{graphicx}
\usepackage{xcolor}
\usepackage{epsfig}
\usepackage{color}
\graphicspath{{Images/}}
\usepackage{tikz}
\usepackage{pgfplots,pgfplotstable}
\usepackage{pgfplots}
\usepgfplotslibrary{groupplots,dateplot}
\pgfplotsset{compat=newest}
\pgfplotsset{every axis/.append style={
		label style={font=\normalsize},
		tick label style={font=\normalsize}  
}}
\usetikzlibrary{patterns,fit,matrix}
\usepackage{tikzscale}
\usepackage{scalerel}
\usetikzlibrary{patterns,
	arrows,
	plotmarks,
	shapes.arrows,
	spy,
	svg.path,
	shapes.multipart,
	positioning,
	decorations,
	decorations.fractals,
	calc,
	decorations.markings,
	matrix}
\usepgfplotslibrary{fillbetween}

\usepackage[linesnumbered,ruled,vlined]{algorithm2e}
\usepackage{multirow}
\usepackage{multicol}
\makeatletter
\gdef\Shortstack{\@ifnextchar[\@Shortstack{\@Shortstack[c]}}
\gdef\@Shortstack[#1]#2{%
	\leavevmode
	\vbox\bgroup
	\baselineskip-\p@\lineskip 3\p@
	\let\mb@l\hss\let\mb@r\hss
	\expandafter\let\csname mb@#1\endcsname\relax
	\let\\\@stackcr\setlength{\baselineskip}{#2}%
	\@ishortstack}
\makeatother
\usepackage{booktabs}
\usepackage{tabularx}

\usepackage{array}
\usepackage[bottom]{footmisc}
\usepackage{stfloats}
\usepackage{textcomp}
\usepackage[absolute]{textpos}

\usepackage{enumitem}

\usepackage{amsmath,amssymb,amsfonts,amsthm}
\usepackage{mathdots}
\usepackage{bbm,dsfont}
\usepackage{relsize}
\usepackage{nicefrac}
\usepackage{bbm}
\usepackage{mathtools}
\usepackage[mathscr]{eucal}
\usepackage[c2,short]{optidef}

\usepackage{cite}

\usepackage{url}
\usepackage[colorlinks,linkcolor=blue,citecolor=blue]{hyperref}
\makeatletter
\let\NAT@parse\undefined
\makeatother
\usepackage[capitalize,nameinlink]{cleveref}
\newcommand\orcidicon[1]{\href{https://orcid.org/#1}{\includegraphics[scale=0.04]{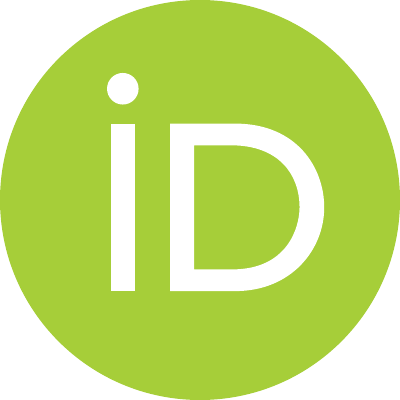}}}

\if0\mode

\else

\fi
\makeatletter
\newcommand{\remove}[1]{\@bsphack\@esphack\xspace}
\newcommand{\arxiv}[1]{\@bsphack\@esphack\xspace}
\makeatother
\newcommand{\extended}[2]{#2\xspace} 

\newcommand{\paperType}{\if0\mode letter\xspace \else paper\xspace \fi}

\newcommand{\bs}[1]{\boldsymbol{#1}}
\newcommand{\cl}[1]{\mathcal{#1}}
\newcommand{\bb}[1]{\mathbb{#1}}

\newcommand{\lb}{\left(}
\newcommand{\rb}{\right)}
\newcommand{\ls}{\left[}
\newcommand{\rs}{\right]}
\newcommand{\lc}{\left\{}
\newcommand{\rc}{\right\}}

\newcommand{\lv}{\left\vert}
\newcommand{\rv}{\right\vert}

\newcommand{\Real}[1]{ { {\mathbb R}^{#1} } }

\DeclareMathOperator*{\argmin}{arg\;min}
\newcommand{\norm}[2][1]{\lVert#2\rVert_{#1}}
\newcommand{\tr}[1]{\operatorname{Tr}\lb#1\rb}
\newcommand{\tril}[1]{\operatorname{Tr}(#1)}
\renewcommand{\ker}[1]{\operatorname{ker}\lc#1\rc}
\newcommand{\keril}[1]{\operatorname{ker}\{#1\}}
\newcommand{\im}[1]{\operatorname{Col}\lc#1\rc}
\newcommand{\imil}[1]{\operatorname{Col}\{#1\}}

\newcommand{\rank}[1]{\operatorname{rk}\lc#1\rc}
\newcommand{\rankil}[1]{\operatorname{rk}\{#1\}}
\newcommand{\proj}[2]{\operatorname{proj}_{#2}\lc#1\rc}
\newcommand{\projil}[2]{\operatorname{proj}_{#2}\{#1\}}
\newcommand{\T}{\top}

\newcommand{\DeclareAutoPairedDelimiter}[3]{%
	\expandafter\DeclarePairedDelimiter\csname Auto\string#1\endcsname{#2}{#3}%
	\begingroup\edef\x{\endgroup
		\noexpand\DeclareRobustCommand{\noexpand#1}{%
			\expandafter\noexpand\csname Auto\string#1\endcsname*}}%
	\x}

\DeclareAutoPairedDelimiter{\ceil}{\lceil}{\rceil}
\DeclareAutoPairedDelimiter{\floor}{\lfloor}{\rfloor}

\newcommand{\e}{\mathrm{e}}

\newcommand{\x}[1]{\bs{x}(#1)}
\renewcommand{\u}[2][]{\bs{u}_{#1}(#2)}

\newcommand{\sch}[1][]{\cl{S}_{#1}}

\newcommand{\reach}[2]{\bs{\Phi}_{#1}^{#2}}
\newcommand{\gram}[2]{\bs{W}_{#1}^{#2}}
\newcommand{\costmin}{\rho_\text{min}}
\newcommand{\costcurr}{\rho_\text{curr}}

\newcommand{\tinit}{T_\text{in}}
\newcommand{\tmin}{T_\text{min}}

\if 0\mode
\newtheoremstyle{colorthm}{\topsep}{\topsep}{\normalfont}{1em}{\color{nblue}\sf\itshape\small\bfseries}{:}{ }{\thmname{#1}\thmnumber{ #2}{\thmnote{ (#3)}}}
\theoremstyle{colorthm}
\else
\theoremstyle{plain}
\fi

\theoremstyle{plain}
\newtheorem{thm}{Theorem}

\newtheorem{prop}{Proposition}
\newtheorem{lemma}{Lemma}

\theoremstyle{definition}
\newtheorem{definition}{Definition}
\newtheorem{prob}{Problem}
\newtheorem{ass}{Assumption}
\newtheorem{rem}{Remark}
\newtheorem{ex}{Example}

\Crefname{thm}{Theorem}{Theorems}
\Crefname{cor}{Corollary}{Corollaries}
\Crefname{conj}{Conjecture}{Conjectures}
\Crefname{ass}{Assumption}{Assumptions}
\Crefname{prop}{Proposition}{Propositions}
\Crefname{figure}{Fig.}{Figures}

\SetCommentSty{mycommfont}

\addto\captionsenglish{\renewcommand{\figurename}{Fig.}}
\addto\captionsenglish{\renewcommand{\tablename}{TABLE}}

\newcommand{\review}[1]{{\color{black}#1}}

\newcommand{\blue}[1]{{\color{black} #1}}
\newcommand{\linkToPdf}[1]{\href{#1}{\blue{(pdf)}}}
\newcommand{\linkToPpt}[1]{\href{#1}{\blue{(ppt)}}}
\newcommand{\linkToCode}[1]{\href{#1}{\blue{(code)}}}
\newcommand{\linkToWeb}[1]{\href{#1}{\blue{(web)}}}
\newcommand{\linkToVideo}[1]{\href{#1}{\blue{(video)}}}
\newcommand{\linkToMedia}[1]{\href{#1}{\blue{(media)}}}
\newcommand{\award}[1]{\xspace} 
\newcommand{\eg}{\emph{e.g.,}\xspace}
\newcommand{\ie}{\emph{i.e.,}\xspace}
\if0\mode
\addto\extrasenglish{}

\addto\extrasenglish{}
\else
\addto\extrasenglish{}

\fi
\addto\extrasenglish{}
\addto\extrasenglish{}

\newcommand{\ptitle}{\titlecap{Pointwise-Sparse Actuator Scheduling for Linear Systems with Controllability Guarantee
		\remove{: Optimized Greedy and MCMC Algorithms}
}}

\if0\mode
\title{\ptitle}
\author{Luca~Ballotta\textsuperscript{\orcidicon{0000-0002-6521-7142}},
	Geethu~Joseph\textsuperscript{\orcidicon{0000-0002-5289-5403}},
	and~Irawati~Rahul~Thete\textsuperscript{\orcidicon{0009-0006-0072-9522}}
	\thanks{This work was supported in part	
		by the EU Horizon program through the project TWAIN, grant id 101122194.
		Views and opinions expressed in this work are of the authors and may not reflect those of the funding institutions.
	}%
	\thanks{Luca Ballotta is with the Delft Center for Systems and Control and Geethu Joseph is with the Signal Processing Systems Group, 
		all at the Delft University of Technology, Netherlands
		(e-mail: \{l.ballotta,g.joseph\}@tudelft.nl).}%
	\thanks{Irawati Rahul Thete is with Texas Instruments, Bengaluru, India. She was with the Signal Processing Systems Group at the Delft University of Technology, Netherlands, during the course of this work.
		(e-mail: irawati.thete@gmail.com).}
}
\else
\title{\ptitle}
\author{Luca Ballotta, Geethu Joseph, and Irawati Rahul Thete
	\thanks{Luca Ballotta is with the Delft Center for Systems and Control, and Geethu Joseph is with the Electrical Engineering, Mathematics, and Computer Science, 
		both at the Delft University of Technology, Delft, Netherlands
		(e-mail: \{l.ballotta,g.joseph\}@tudelft.nl).}%
	\thanks{Irawati Rahul Thete is with Texas Instruments, Bengaluru, India
		(e-mail: irawati.thete@gmail.com).}
}
\fi

\begin{document}
	
	\bstctlcite{MyBSTcontrol}
	
	\maketitle
	
	\if 0\mode
	\pagestyle{empty}
	\thispagestyle{empty}
	\fi
	

\begin{abstract}
	\review{This \paperType considers the design of sparse actuator schedules for linear time-invariant systems.
    	An actuator schedule selects,
    	for each time instant,
    	which control inputs act on the system in that instant.}
    We address the optimal scheduling of control inputs under a hard constraint on the number of inputs that can be used at each time.
    For a sparsely controllable system,
    we characterize sparse actuator schedules that make the system controllable,
    and then devise a greedy selection algorithm that guarantees controllability while heuristically providing low control effort.
    We further show how to enhance our greedy algorithm via Markov chain Monte Carlo-based randomized optimization. 
	
	\if 1\mode
	\fi
	
\end{abstract}

\section{Introduction}\label{sec:intro}

Sparsity constraints in control inputs arise in several large-scale systems\extended{,
such as networked control with limited bandwidth~\cite{Heemels10tac-networked}, budget-limited influence in marketing strategies~\cite{Joseph21siam-controllability},  sparse damping control in power grids~\cite{Siami21tac-actuatorScheduling}, and drug treatment in biological networks~\cite{Rajapakse12plos-can}.}
{. For example, sparse control inputs admit compact representations~\cite{foucart13-invitation}, conserving bandwidth in networked control systems~\cite{Heemels10tac-networked,Nagahara16-MaximumHands-OffControl}. In large-scale social networks, influencers or marketers promote their products or ideas by influencing a few (sparse) individuals and relying on word-of-mouth effects~\cite{Liu15InfoSc-Identifying,Joseph21siam-controllability}. Similarly, sparse high-voltage DC line scheduling for wide-area damping control in electrical power grids stabilizes fluctuations and synchronizes generators with reduced costs and energy depletion~\cite{Siami21tac-actuatorScheduling}. In biological networks characterized by model reactions and/or metabolites, drugs target a small number of nodes via sparse control to minimize adverse side effects~\cite{Rajapakse12plos-can,Olshevsky14tcns-minimalControllability}.}
Motivated by these applications, we address the challenge of designing control with fewer actuators without significantly compromising performance.

\extended{Two important sparse control paradigms are maximum hands-off control,
	which reduces active (nonzero) control periods~\cite{nagahara2020sparsity},
	and sparse feedback,
	which reduces the nonzero elements or rows in the feedback gain matrix~\cite{Jovanovic16ejc-ControllerArchitectures,Ballotta23tcns-decentralizedVsCentralized}.}{Sparse input strategies have existed in control theory for many years~\cite{athans1972determination},
	especially motivated by networked and large-scale systems, 
	but gained prominence following the emergence of compressed sensing~\cite{foucart13-invitation}.
	Closely related to sparse actuator control, sparsity-promoting optimization of the feedback gain matrix has been proposed as an effective method to trade controller complexity for performance, with actuator selection as a special case that enjoys convexity~\cite{Jovanovic16ejc-ControllerArchitectures,Bahavarnia17ifacwc-sparseLQR}. Furieri \emph{et al.}~\cite{Furieri20tcns-sparsityInvariance} studied a convex problem formulation based on pre-defined sparsity patterns. Matni and Chandrasekaran~\cite{Matni16tac-regularizationForDesign} proposed a regularization for design that can encode various constraints, such as communication locality. Anderson \emph{et al.}~\cite{Anderson19arc-sls} introduced the design framework System Level Synthesis, where sparsity patterns can be imposed directly on the closed-loop matrices. 
	The authors in~\cite{Ballotta23tcns-decentralizedVsCentralized,Ballotta23lcss-fasterConsensus} revealed a fundamental performance tradeoff arising with architecture-dependent delays and that sparse feedback can be optimal.
	Another important control paradigms focused on sparse design is maximum hands-off control, 
	which reduces active (nonzero) control periods~\cite{ikeda2018sparsity,ito2021sparse,nagahara2020sparsity}.}%
Some works have investigated sparsity directly in actuator use. 
An early approach selects a few inputs,
which remain constant over time,
to ensure controllability, 
tightly constraining the system~\cite{Olshevsky14tcns-minimalControllability}. 
Recent work minimizes the average number of inputs over a time horizon, 
allowing for time-varying inputs\cite{Siami21tac-actuatorScheduling}. 
However, this strategy can lead to non-sparse inputs at certain times,
which is unsuitable for applications like networked systems where bandwidth constraints must be satisfied at all times.
Alternatively, this work limits the number of control inputs active at every time~\cite{Joseph21tac-sparseControllability}, referred to as \emph{sparse actuator control}.

The theoretical foundation of sparse actuator control\extended{}{, 
covering the necessary and sufficient conditions for controllability and stabilizability,}
has been studied in \cite{Joseph21tac-sparseControllability,Joseph23tcns-output,Joseph22tac-Controllability,Sriram23tac-sparseStabilizability}, but sparse control input design is not well studied in the literature. 
A straightforward method uses sparse recovery algorithms to drive the system from a given state to a desired state~\cite{Sriram23tac-sparseStabilizability},
which requires distinct designs for state deviations. 
The standard,
naive greedy algorithm identifies a sparse actuator schedule\extended{}{ that enables a sparse control input sequence following the schedule} to transition the system from any initial state to any desired state~\cite{kondapi2024sparseactuatorschedulingdiscretetime}. 
Due to its heuristic nature, it does not assure controllability for all systems that are controllable under sparsity constraints. 
We address this literature gap and propose efficient algorithms for the design of sparse control inputs that formally guarantee controllability.

\subsubsection*{Contribution}
We first show that the naive greedy algorithm may fail to ensure controllability (\autoref{sec:design-greedy}),
which motivates our study. 
Then, we analytically characterize sparse actuator schedules that make the system controllable (\autoref{sec:feasible-set}) and devise an improved greedy algorithm that searches among such schedules (\autoref{sec:design-s-greedy}),
thus guaranteeing controllability if the system is controllable under the sparsity constraint. 
Also, we propose to improve the algorithm using a Markov chain Monte Carlo (MCMC)-based approach,
albeit with increased computational complexity (\autoref{sec:design-mcmc}).
Finally, in \autoref{sec:experiments}, we numerically study the performance of our greedy algorithm on two use cases where the naive greedy algorithm fails, as well as the improvement provided by MCMC.

\section{Sparse Actuator Scheduling Problem}\label{sec:setup}

Consider the discrete-time linear dynamical system $(\bs{A},\bs{B})$ with matrices $\bs{A}\in\Real{n\times n}$,
$\bs{B}\in\Real{n\times m}$ and state evolution
\begin{equation}\label{eq:systemmodel}
	\x{k+1} = \bs{A}\x{k} + \bs{B}\u{k},
\end{equation}
where $\x{k}\in\Real{n}$ and $\u{k}\in\Real{m}$ respectively denote state and control input at time $k$. 
The system $(\bs{A},\bs{B})$ is $s$-sparse controllable if it is controllable when $\norm[0]{\u{k}}\le s\ \forall k\ge0$,
\ie when at most $s$ input channels
are active (nonzero) at all times, where a channel is an element of $\u{k}$. 
We aim to design sparse inputs to drive the system from any initial state to a desired one.
We first review the existing literature.

\subsection{Preliminaries on Sparse Controllability}
\label{sec:preliminaries}

A simple test for $s$-sparse controllability is as follows,
\review{where $\rank{\bs{A}}$ denotes the rank of matrix $\bs{A}$.}
\begin{prop}[\!\!\!\protect{\cite[Theorem~1]{Joseph21tac-sparseControllability}}]\label{thm:sparse-controllability}
	System~\eqref{eq:systemmodel} is $s$-sparse controllable iff it is controllable and $s \ge \max\{n-\rank{\bs{A}}, 1\}$.
\end{prop}

\review{The next result ensures $s$-sparse controllability with fixed active input channels, following an \emph{actuator schedule},
regardless of the initial and desired states.}

\review{
	\begin{definition}
		Let $[m]\doteq\{1,\ldots,m\}$ with $[0]\doteq\emptyset$.
		An \emph{actuator schedule} $\sch$ on $m$ inputs with horizon $h$ is an ordered tuple
		\begin{equation}
			\sch=(\sch[0],\dots,\sch[h-1]): \;\sch[k]\subseteq[m] \ \forall  k=0,\dots,h-1.
		\end{equation}
		We denote the set of all actuator schedules by 		$\mathcal{T}_m^h$.
	\end{definition}
}
In words,%
\extended{}{ an actuator schedule gathers the input channels that are active over the horizon.
Specifically,}
\review{the $i$th input channel $\u[i]{k}$ is active at time $k$ only if $i\in \sch[k]$.}
Also,
for given sparsity $s\in[m]$,
an \emph{$s$-sparse actuator schedule} fulfills the condition $|\sch[k]|\leq s$ for each time $k$,
\review{meaning that at most $s$ inputs are active at every time.}

\extended{Under}{If the $i$th entry of $\u{k}$ is zero, then the $i$th column of $\bs{B}$ does not affect the state $\x{k+1}$.
Therefore,
under}
the actuator schedule $\sch$, the actual state evolution is 
\begin{equation}\label{eq:systemmodel-sparse}
	\x{k+1} = \bs{A}\x{k} + \bs{B}_{\sch[k]}\u[{\sch[k]}]{k}, \quad k=0,\dots,h-1,
\end{equation}
where $\bs{B}_{\cl{I}}$ denotes the submatrix of $\bs{B}$ composed by columns with index in set $\cl{I}$ and $\u[{\sch[k]}]{k}$ stacks the elements of $\u{k}$ with index in $\cl{I}$.
We define the \emph{$\sch$-reachability matrix} as
\begin{equation}\label{eq:reach-mat}
	\reach{\sch}{h} \doteq \begin{bmatrix}
		\bs{A}^{h-1}\bs{B}_{\sch[0]} & \bs{A}^{h-2}\bs{B}_{\sch[1]} & \ldots & \bs{B}_{\sch[h-1]}
	\end{bmatrix}.
\end{equation}
\begin{prop}[\!\!\protect{\cite[Theorems~2.4 and 2.13]{joseph2024sparse}}]
\label{thm:sparse-schedule}
	If system~\eqref{eq:systemmodel} is $s$-sparse controllable and $h\ge h^*$,
	where
	 \begin{equation}
		\frac{n}{\min\{\rank{\bs{B}},s\}}\leq h^*\leq n-\min\{\rank{\bs{B}},s\}+1,
	\end{equation}
	 there exists an $s$-sparse schedule $\sch\in\cl{T}_m^h$ that makes system~\eqref{eq:systemmodel-sparse} $s$-sparse controllable,
	 \ie for which $\rankil{\reach{\sch}{h}} = n$\extended{.}{, where $\reach{\sch}{h}$ is defined in \eqref{eq:reach-mat}.}
\end{prop}
Using the above insights, we formulate our problem next.

\subsection{Sparse Control Design Problem Formulation}\label{sec:problem-formulation}

\extended{In light of \Cref{thm:sparse-schedule},
	we aim to find an $s$-sparse actuator schedule $\sch$ that makes the resulting system~\eqref{eq:systemmodel-sparse} controllable.
}{Designing sparse control inputs is challenging due to the nonconvex, combinatorial nature of sparsity constraints. However, as shown in \Cref{thm:sparse-schedule}, there exists a sparse actuator schedule that ensures controllability and is independent of the system's initial and final states. 
	We address sparsity by devising an actuator schedule that determines the support of the control input, thus removing the sparsity constraint. 
	During operation, 
	the control inputs can be designed using the least squares method.
	Thus, our goal is to find a schedule $\sch=(\sch[0],\sch[1],\ldots,\sch[h-1])\in\mathcal{T}_m^h$ such that $|\sch[k]|\le s$ and $\rankil{\reach{\sch}{h}} = n$, where $\reach{\sch}{h}$ is defined in \eqref{eq:reach-mat}.}

\extended{To}{A sparse actuator schedule satisfying controllability need not be unique, so we introduce a metric to} measure the control effort of a \extended{schedule $\sch$}{given schedule. To this end}, 
we consider the \emph{$\sch$-controllability Gramian} $\gram{\sch}{h} = (\reach{\sch}{h})^{\top}\reach{\sch}{h}$\extended{, which}{. We note that $\rankil{\gram{\sch}{h}}=\rankil{\reach{\sch}{h}}$. Also, $\gram{\sch}{h}$} \review{is used in the literature to quantify} the control energy required to reach a target state through~\eqref{eq:systemmodel-sparse} in $h$ steps~\cite{Baggio22ar-energyAwareControllability}.
\review{We denote the control metric by $\rho$ and list common energy-aware performance metrics in \autoref{table:metrics}.}
Therefore, our design problem is as follows.

\begin{table}
	\caption{Common energy-aware control  performance metrics~\cite{Baggio22ar-energyAwareControllability}.}
	\label{table:metrics}
	\begin{center}
		\small
		\begin{tabular}{lll}
			\toprule
			$\frac{1}{n}\tril{\gram{}{-1}}$	& average energy to reach a unit-norm state\\
			$(\lambda_{\text{min}}(\gram{}{}))^{-1}$	& maximal energy to reach a unit-norm state\\
			$\sqrt[n]{\det(\gram{}{})}$	& volume of ellipsoid reached with unit energy\\
			\bottomrule
		\end{tabular}
	\end{center}
\end{table}

\begin{prob}[Optimal sparse actuator schedule]\label{prob:design}
	Given system~\eqref{eq:systemmodel},
	a time horizon $h$,
	a sparsity level $s$,
	and a cost function $\rho$,
	find an $s$-sparse schedule $\sch^*$ as 	
	\begin{argmini!}
		{\sch\in\mathcal{T}_m^h}
		{\rho\lb\gram{\sch}{h}\rb\protect\label{eq:prob-design-obj}}
		{\label{eq:prob-design}}
		{\sch^*\in}
		\addConstraint{|\sch[k]|}{\le s}{\forall k\in\{0,\dots,h-1\}\protect\label{eq:prob-design-constr-sparsity}}
		\addConstraint{\rank{\gram{\sch}{h}}}{= n.\protect\label{eq:prob-design-constr-controllability}}
	\end{argmini!}
\end{prob}

In problem~\eqref{eq:prob-design},
the combinatorial constraint~\eqref{eq:prob-design-constr-sparsity} makes the schedule 
(\ie the control inputs) point-wise $s$-sparse,
and the rank constraint~\eqref{eq:prob-design-constr-controllability} ensures that the chosen schedule makes the system $s$-sparse controllable.%
\extended{%
}{ Problem~\eqref{eq:prob-design} is combinatorial in nature and can be solved by exhaustive search, whose computational complexity does not scale with the sparsity constraint $s$ and the horizon $h$.
In fact, the authors in~\cite{Siami21tac-actuatorScheduling} suggest that~\eqref{eq:prob-design} is NP-hard, although no proof is given.}%
\extended{\cref{thm:sparse-controllability,thm:sparse-schedule} imply that problem~\eqref{eq:prob-design} is feasible if the following conditions hold.
	\begin{ass}\label{rem:feasibility}
		The sparsity level and the time horizon respectively satisfy $s\geq \max\{n-\rank{\bs{A}},1\}$ and $h\geq h^*$.
	\end{ass}
}{Next, we provide the conditions under which the problem is feasible using \cref{thm:sparse-controllability,thm:sparse-schedule}.
	\begin{rem}\label{rem:feasibility}
		Given matrices $\bs{A}\in\Real{n\times n}$ and $\bs{B}\in\Real{n\times m}$,
		a time horizon $h$,
		and a sparsity constraint $s$,
		\cref{prob:design} is feasible if the system $(\bs{A},\bs{B})$ is controllable,
		the sparsity level satisfies $s\geq \max\{n-\rank{\bs{A}},1\}$,
		and the time horizon $h\geq h^*$.
\end{rem}}

In the following, we assume that \cref{rem:feasibility} holds and present design algorithms with polynomial complexity to solve \cref{prob:design} ensuring the constraints~\eqref{eq:prob-design-constr-sparsity}--\eqref{eq:prob-design-constr-controllability} are met.

\section{Design Algorithms}\label{sec:design}
 This section tackles \cref{prob:design} \review{via algorithms with tractable (polynomial) computational complexity}.
 First,
 we explain why the naive greedy algorithm may fail to meet the constraints in~\eqref{eq:prob-design}.
 In~\autoref{sec:feasible-set} we characterize sparse schedules that ensure controllability and,
 building on the analysis, 
 in~\autoref{sec:design-s-greedy} we devise our main design algorithm with formal controllability guarantees.
 Finally,
 in \autoref{sec:design-mcmc} we propose a stochastic optimization approach to improve performance.

\subsection{Drawback of the Naive Greedy Algorithm}\label{sec:design-greedy}

A classic approach to reduce cost~\eqref{eq:prob-design-obj} subject to a budget constraint on the inputs is greedy selection,
which we refer to as a naive greedy algorithm~\cite{Baggio22ar-energyAwareControllability,kondapi2024sparseactuatorschedulingdiscretetime}.%
\extended{}{ It selects columns from the reachability matrix $\reach{}{h}$ until the budget is exhausted or the cost cannot be reduced.
This method typically performs well in practice and enjoys quantifiable suboptimality guarantees if $\rho$ is supermodular~\cite{Summers16tcns-submodularity,Olshevsky18tcns-nonSupermodularity}.
Also, the seminal work~\cite{Olshevsky14tcns-minimalControllability} shows that greedy selection enjoys the lowest computational complexity required to find the tightest approximate minimal number of control inputs that ensure controllability in polynomial time.}%
However,
imposing the point-wise $s$-sparsity constraint in~\eqref{eq:prob-design-constr-sparsity} to the greedy selection may yield an uncontrollable system, as in the example below.

\begin{ex}\label{ex}
	We choose a system with $n=5$, $m=7$ where
	\begin{equation}\label{eq:ex}
		\bs{A} = \begin{bmatrix}
			0 & 1 & 0 & 0 & 0\\
			0 & 0 & 0 & 1 & 0\\
			0 & 0 & 1 & 0 & 0 \\
			0 & 0 & 0 & 0 & 1 \\
			0 & 0 & 0 & 0 & 0
		\end{bmatrix},
		\bs{B} = \begin{bmatrix}
			0	&	0	&	1	&	0	&	0	&	0	&	1\\
			0	&	0	&	1	&	0	&	0	&	1	&	0\\
			1 	&	0	&	0	&	0	&	1	&	0	&	1\\
			1	&	1	&	0	&	0	&	0	&	0	&	1\\
			0	&	0	&	0	&	1	&	0	&	0	&	0
		\end{bmatrix}.
	\end{equation}
	The system is $s$-sparse controllable with $s=n-\rankil{\bs{A}}=1$.
	We choose time horizon $h = n=5$ and cost function $\rho(\cdot) = \tr{(\cdot)^{-1}}$.
	To avoid singular matrix inversion, the cost is computed as $\operatorname{Tr}((\gram{\sch}{h})^{-1}+\epsilon\bs{I})$ with a small slack $\epsilon>0$~\cite{Siami21tac-actuatorScheduling,kondapi2024sparseactuatorschedulingdiscretetime}.%
	\extended{}{ To trade robustness for accuracy, we start with a default slack $\epsilon = 10^{-10}$ that is progressively increased till the cost can be computed.}%
	In the first iteration, the last column of the reachability matrix $\reach{}{h}$ gives the lowest cost, so the algorithm selects $\sch[4] = 7$. 
	Due to the sparsity constraint, subsequent iterations restrict the greedy search to the first $28$ columns of $\reach{}{h}$.
	However, the last row of these columns is entirely zero. 
	Thus, the last row of $\reach{\sch}{h}$	has only zeros,
	yielding an uncontrollable system and failure of the naive greedy algorithm.
\end{ex}

In short, the naive greedy algorithm need not output a feasible solution of the control design problem~\eqref{eq:prob-design}. 
We next examine the feasible set of~\eqref{eq:prob-design} to guide a greedy selection.

\subsection{Characterizing Feasible Actuator Schedules}
\label{sec:feasible-set}

We start with a necessary condition for feasibility.

\begin{lemma}
	Let $\sch$ be a feasible solution of \cref{prob:design} satisfying \cref{rem:feasibility}.  
	Then, for $k=1,\ldots,h-1$,
	it holds
	\begin{equation}\label{eq:rank_condition}
		\rank{\begin{bmatrix}
				\bs{A}^{k}& \bs{A}^{k-1}\bs{B}_{\sch[h-k]} & \ldots & \bs{B}_{\sch[h-1]}
		\end{bmatrix}} = n. 
	\end{equation}
\end{lemma}
\begin{proof}
	From the rank condition, we deduce that
	\begin{equation}
		\begin{aligned}
			n&=\rank{\begin{bmatrix}
					\bs{A}^{h-1}\bs{B}_{\sch[0]} & \ldots & \bs{B}_{\sch[h-1]}
			\end{bmatrix}}\\
			&\leq \rank{
				\begin{bmatrix}\bs{A}^{k}& \bs{A}^{k-1}\bs{B}_{\sch[h-k]} & \ldots & \bs{B}_{\sch[h-1]}
			\end{bmatrix}} \leq n,
		\end{aligned}
	\end{equation} 
 and the equality~\eqref{eq:rank_condition} immediately follows by comparison.
\end{proof}

From the relation~\eqref{eq:rank_condition} with $k=1$, we derive
\begin{equation}
	\bb{R}^n=\im{\begin{bmatrix}
			\bs{A} & \!\!\!\!\bs{B}_{\sch[h-1]}
	\end{bmatrix}}= \im{\bs{A}}\oplus\im{\bs{B}_{\sch[h-1]}},
\end{equation}
where $\im{\cdot}$ denotes the column space and $\oplus$ denotes the sum of subspaces.
As $
	\bb{R}^n=\im{\bs{A}}  \oplus \keril{\bs{A}^{\top}}  = \im{\bs{A}}\oplus\im{\bs{B}_{\sch[h-1]}}$,
where $\ker{\cdot}$ denotes the null space, we get
\begin{equation}\label{eq:nullA}
	\ker{\bs{A}^{\top}}  = \proj{\im{\bs{B}_{\sch[h-1]}}}{\ker{\bs{A}^{\top}}},
\end{equation}
where $\proj{\cl{D}}{\cl{C}}$ denotes the projection of the (sub)space $\cl{D}$ onto the subspace $\cl{C}$. 
For $k=2$,
we write the relation~\eqref{eq:rank_condition} as
\begin{align}\label{eq:nullA2}
	\begin{split}
		\ker{\lb\bs{A}^{2}\rb^\top}  = &\proj{\im{\bs{A}\bs{B}_{\sch[{h-2}]}}}{\ker{\lb\bs{A}^{2}\rb^\top}} \\
		&\oplus \proj{\im{\bs{B}_{\sch[{h-1}]}}}{\ker{\lb\bs{A}^{2}\rb^\top}}.
	\end{split}
\end{align}
Since $\keril{\bs{A}^{\top}}\subseteq\keril{(\bs{A}^{2})^{\top}}$, \eqref{eq:nullA} and \eqref{eq:nullA2} jointly yields
\begin{multline}
	\ker{\lb\bs{A}^{2}\rb^\top} \ominus \ker{\bs{A}^\top} = 
	\\
	\begin{aligned}
		&\proj{\im{\bs{A}\bs{B}_{\sch[{h-2}]}}}{\ker{\lb\bs{A}^{2}\rb^{\top}} {\ominus \ker{\bs{A}^{\top}}}} \\
		&\oplus \proj{\im{\bs{B}_{\sch[{h-1}]}}}{\ker{\lb\bs{A}^{2}\rb^{\top}} {\ominus \ker{\bs{A}^{\top}}}},
	\end{aligned}
\end{multline}
where
for any two subspaces $\cl{C}\subseteq\cl{D}$,
$\cl{D}\ominus\cl{C}\subseteq\cl{D}$ denotes the orthogonal complement to $\cl{C}$ in $\cl{D}$. 
The dimension of the subspace $\keril{(\bs{A}^{2})^{\top}} {\ominus \keril{\bs{A}^{\top}}}$ is the difference between the dimensions of $\keril{(\bs{A}^{2})^{\top}}$ and $\keril{\bs{A}^{\top}}$, given by 
\begin{equation}
	(n-\rank{\bs{A}^2})-(n-\rank{\bs{A}})
	\leq n-\rank{\bs{A}}\leq s
\end{equation}
where we use the Sylvester rank inequality.
Hence,
regardless of the channels in $\sch[h-1]$,
we can find a set $\sch[{h-2}]$ such that 
\begin{multline}
	\ker{\lb\bs{A}^{2}\rb^\top} \ominus \ker{\bs{A}^\top} = \\
	\proj{\im{\bs{A}\bs{B}_{\sch[{h-2}]}}}{\ker{\lb\bs{A}^{2}\rb^{\top}} {\ominus \ker{\bs{A}^{\top}}}}.
\end{multline}
Extending the same idea, 
we obtain sufficient conditions for controllability.
To this aim,
we define
\begin{equation}\label{eq:R}
	R \doteq \min\{ k\geq 0:  \rankil{\bs{A}^{k}}=\rankil{\bs{A}^{k+1}}\}.
\end{equation}

\begin{lemma}\label{lem:suff}
    Let $s$ and $h$ satisfy \Cref{rem:feasibility}.
	Then, for all $k\in[R]$, 
	there exists an index set $\sch[h-k]$ such that 
	\begin{multline}\label{eq:suff_feasibility}
		\ker{\lb\bs{A}^{k}\rb^\top} \ominus \ker{\lb\bs{A}^{k-1}\rb^\top} = \\
		\proj{\im{\bs{A}^{k-1}\bs{B}_{\sch[{h-k}]}}}{\ker{\lb\bs{A}^{k}\rb^\top} \ominus \ker{\lb\bs{A}^{k-1}\rb^\top}}.
	\end{multline}
	Further, if $R>0$,
	the sets satisfy
	\begin{multline}\label{eq:span_null}
		\ker{\lb\bs{A}^R\rb^{\top}} = \operatorname{proj}_{\ker{\lb\bs{A}^R\rb^{\top}}}\Big\{ \operatorname{Col}\Big\{
		\big[
		\bs{A}^{R-1}\bs{B}_{\sch[h-R]}
		\\
		\begin{matrix}\bs{A}^{R-2}\bs{B}_{\sch[h-R+1]} & \ldots & \bs{B}_{\sch[h-1]}\end{matrix}\big]\Big\}\Big\}.
	\end{multline}
\end{lemma}
\begin{proof}
	See~\Cref{app:suff}.
\end{proof}

If $\bs{A}$ is invertible,
then $R=0$ and the result is trivial.
\cref{lem:suff} leads to a sufficient condition for feasibility, as established by the following result. 

\begin{thm}\label{thm:kernel-condition}
	Consider system~\eqref{eq:systemmodel} and sparsity level $s$ and time horizon $h$ satisfying \Cref{rem:feasibility}.
	Then, there exists a feasible solution $\cl{S}$ of \cref{prob:design} that satisfies \eqref{eq:suff_feasibility} for $k\in[R]$.
\end{thm}
\begin{proof}
	See \Cref{app:kernel-condition}.
\end{proof}

\extended{}{\cref{thm:kernel-condition} provides two key insights for design.  
	One, it shows that the sparse controllability index, \ie the schedule length, is at least $R$. Second, it establishes the structure of the schedule $\cl{S}_k$ for $k=h-R,\ldots,h-1$.}
Based on \cref{thm:kernel-condition}, we devise \review{an improved greedy algorithm to find a schedule with formal controllability guarantee.}

\subsection{Improved $s$-Sparse Greedy Algorithm}\label{sec:design-s-greedy}

\begin{algorithm}[t]
	\caption{$s$-sparse greedy selection}
	\label{alg:sparse-greedy}
	\KwIn{Matrices $\bs{A},\bs{B}$, sparsity $s$, horizon $h$, cost $\rho$.}
	\KwOut{Schedule $\mathcal{S}$.}
	$\mathcal{S} \leftarrow [[m]] * h$ \tcp*{initialize full schedule}\label{alg:s-greedy-init}
	$\texttt{rk}_W \leftarrow 0$ \tcp*{rank of controllability Gramian}
	$\bs{C} \leftarrow [ \ ]$ \tcp*{column space of controllability Gramian}
	$\texttt{cand} \leftarrow  [\emptyset] * h$ \tcp*{channels to be selected later}
	\For(\tcp*[f]{input at $k=0$ is unconstrained}){$k = 1,\dots,h-1$\label{alg:s-greedy-loop-kernel}}{
		$\sch[k]\leftarrow\emptyset$ \tcp*{select necessary channels at time $k$}
		\If{$\texttt{rk}_W < n$}{
			$\cl{K}_k \leftarrow \keril{(\bs{A}^{h-k})^\top} \ominus \keril{(\bs{A}^{h-1-k})^\top}$\;\label{alg:s-greedy-kernel}
			$\sch[k] \leftarrow \texttt{greedy\_k}(\{j : \bs{B}_j \not\perp \cl{K}_k\},k,\sch,\epsilon)$\;\label{alg:s-greedy-kernel-greedy}
			$\texttt{rk}_W \leftarrow \texttt{rk}_W + |\sch[k]| $\;
			$\bs{C} \leftarrow [\bs{A}^{h-1-k}\bs{B}_{\sch[k]} \; \bs{C}]$\;
		}
		\If{$|\sch[k]|<s$}{
			$\texttt{cand}[k] \leftarrow [m] \setminus \sch[k]$\;\label{alg:s-greedy-loop-kernel-end}
		}
	}
	$\sch[0] \leftarrow \emptyset$; $\texttt{cand}[0] \leftarrow [m]$ \tcp*{reset schedule at time $0$}
	\For(\tcp*[f]{ensure controllability}){\texttt{i} = $1,\dots,n-\texttt{rk}_W$\label{alg:s-greedy-increase-rank}}{
		$\texttt{cand}, \mathcal{S} \leftarrow$ \texttt{greedy}$(\texttt{cand}, \mathcal{S}, \epsilon, \texttt{rk})$\;\label{alg:s-greedy-increase-rank-end}
	}
	\While(\tcp*[f]{decrease cost}){$\exists k : |\sch[k]| < s$\label{alg:s-greedy-drop-cost}}{
		$\texttt{cand}, \mathcal{S} \leftarrow$ \texttt{greedy}$(\texttt{cand}, \mathcal{S},0)$\;
		\If{$\sch$ doesn't change}{
			\textbf{break}\;\label{alg:s-greedy-drop-cost-end}
		}
	}
	\textbf{return} $\mathcal{S}$.
\end{algorithm}

We outline our selection procedure in \cref{alg:sparse-greedy}.
The schedule is initialized full,
\ie all input channels scheduled at all times (\cref{alg:s-greedy-init}).
This allows us to \review{reduce the risk of dealing with singular $\sch$-controllability Gramian and thus to compute the cost function more robustly in the first phase of the algorithm},
which ensures that all channels that are necessary for controllability are scheduled,
This phase is implemented in the ``for'' loop at \cref{alg:s-greedy-loop-kernel}.
Specifically,
for each time step $k>1$,
\cref{alg:s-greedy-kernel} computes the subspace $\cl{K}_k$ that has to be spanned by input channels active at time $k$,
and \cref{alg:s-greedy-kernel-greedy} greedily selects the best channels among the candidates.
For example,
at time $k=h-1$,
$\cl{K}_{h-1}=\keril{\bs{A}^\top}$ and \cref{alg:s-greedy-kernel-greedy} populates $\sch[h-1]$ until condition~\eqref{eq:nullA} is met.
The feasibility of this selection is formally supported by \cref{thm:kernel-condition},
which states that one can find indices $\sch[k]$ such that $\bs{B}_{\sch[k]}$ spans $\cl{K}_k$ for every time $k$.

The selection at \cref{alg:s-greedy-kernel-greedy} of \cref{alg:sparse-greedy} is executed by the subroutine \texttt{greedy\_k} outlined in \cref{alg:greedy-k},
where we define the ``contribution'' to the controllability Gramian of channel $c$ scheduled at time $k$ as $\bs{\phi}_c^k\doteq (\bs{A}^{h-1-k}\bs{B}_c)^\top\bs{A}^{h-1-k}\bs{B}_c$.
Here,
after a channel is selected (\cref{alg:greedy-k-selection}),
the candidate channels that are not independent of those already selected are removed (\cref{alg:greedy-k-if,alg:greedy-k-if-end}),
guaranteeing that the schedule $\sch[k]$ spans one more direction of $\cl{K}_k$ after each iteration of the ``while'' loop of \cref{alg:greedy-k} and eventually spans all of $\cl{K}_k$ so as to fulfill~\eqref{eq:rank_condition}.
By the initialization of $\sch$,
the parameter $\epsilon$ in \cref{alg:greedy-k} can be zero if $\cl{K}_k$ has dimension one,
but it must be positive otherwise to avoid a singular matrix in $\rho$.

\begin{algorithm}[t]
	\caption{Subroutine \texttt{greedy\_k}}
	\label{alg:greedy-k}
	\KwIn{Channels \texttt{cand}, step $k$, schedule $\sch$, $\epsilon\ge0$.}
	\KwOut{Updated schedule $\sch[k]$ at time step $k$.}
	\While{$\texttt{cand}\neq\emptyset$}{
		$c^* \leftarrow  \argmin_{c \in \texttt{cand}} \rho(\gram{\sch}{h} + \bs{\phi}_c^k + \epsilon \bs{I})$\;
		$\sch[k] \leftarrow \sch[k] \cup \{c^*\}$; $\texttt{cand} \leftarrow \texttt{cand} \setminus \{c^*\}$\;\label{alg:greedy-k-selection}
		\ForEach{$c\in\texttt{cand}$}{
			\If{$\projil{\bs{A}^{h-1-k}\bs{B}_{c}}{\cl{K}_k} \parallel \projil{\bs{A}^{h-1-k}\bs{B}_{\sch[k]}}{\cl{K}_k}$\label{alg:greedy-k-if}}{
				$\texttt{cand} \leftarrow \texttt{cand} \setminus \{c\}$\;\label{alg:greedy-k-if-end}
			}
		}
	}
	\textbf{return} $\sch[k]$.
\end{algorithm}

\begin{algorithm}[b]
	\caption{Subroutine \texttt{greedy}}
	\label{alg:greedy}
	\KwIn{Channels \texttt{cand}, schedule $\sch$, $\epsilon\ge0$, flag \texttt{rk}.}
	\KwOut{Updated \texttt{cand}, $\sch$.}
	$ k^*, c^* \leftarrow  \argmin_{k : |\sch[k]| < s,c \in \texttt{cand}[k]} \rho(\gram{\sch}{h} + \bs{\phi}_c^k + \epsilon \bs{I})$\;\label{alg:greedy-selection}
	\If{\textbf{not} $\!\!\!$ \texttt{rk} \textbf{and} $\rho(\gram{\sch}{h} + \bs{\phi}_{c^*}^{k^*} + \epsilon \bs{I}) \!\ge\! \rho(\gram{\sch}{h} + \epsilon \bs{I})$}{
		\textbf{return} \texttt{cand}, $\sch$\;
	}
	$\sch[k^*] \leftarrow \sch[k^*] \cup \{c^*\}$; $\texttt{cand}[k^*] \leftarrow \texttt{cand}[k^*] \setminus \{c^*\}$\;
	\If{\texttt{rk}\label{alg:greedy-rk}}{
		$\bs{C} \leftarrow [\bs{A}^{h-1-k^*}\bs{B}_{c^*} \; \bs{C}]$\;
		\ForEach{$k : |\sch[k]| < s$, $c\in\texttt{cand}[k]$}{
			\If{$\rankil{[\bs{A}^{h-1-k}\bs{B}_{c} \; \bs{C}]} = \rank{\bs{C}}$\label{alg:greedy-if}}{
				$\texttt{cand}[k] \leftarrow \texttt{cand}[k] \setminus \{c\}$\;\label{alg:greedy-if-end}
			}
		}
	}
	\textbf{return} \texttt{cand}, $\sch$.
\end{algorithm}

After including all necessary input channels to span the kernels of $\bs{A},\bs{A}^2,\dots,\bs{A}^{R}$ so as to fulfill \Cref{thm:kernel-condition},
\cref{alg:sparse-greedy} ensures controllability through the selected schedule with the ``for'' loop at \cref{alg:s-greedy-increase-rank}.
This is achieved via the subroutine \texttt{greedy} in \cref{alg:greedy},
which is the classic greedy selection but with the option to check the rank of the controllability Gramian through the flag \texttt{rk}.
When \texttt{rk} is raised (\cref{alg:greedy-rk}),
all candidate channels that cannot increase the rank of the controllability Gramian accrued so far are removed at \cref{alg:greedy-if,alg:greedy-if-end}.
This ensures that the ``for'' loop at \cref{alg:s-greedy-increase-rank} of \cref{alg:sparse-greedy} eventually makes the $\sch$-controllability Gramian full rank.

Finally,
the last ``while'' loop at \cref{alg:s-greedy-drop-cost} adds channels till the cost cannot be further decreased or the selected schedule has reached the $s$-sparsity limit.\footnote{
	Because the greedy selection is suboptimal,
	if the cost cannot be decreased in one iteration,
	one can randomly select channels to possibly drop the cost  after  multiple iterations.
	We do not explore this possibility in the current \paperType and leave a more comprehensive numerical evaluation for future work.
}
Notably,
this last selection need not care about the rank of the controllability Gramian,
and the cost can be computed exactly by setting $\epsilon=0$.

\extended{\review{Assuming that computing the cost $\rho(\gram{\sch}{h})$ requires $O(n^\beta)$ operations,
	the computational complexity of \Cref{alg:sparse-greedy} is $O(hm\max\{n^{\beta+1},n^4)\}$,
	the same of the naive greedy.
	The detailed analysis is provided in the technical report~\cite{arxiv}.
}}{}
			\extended{}{

\subsubsection*{Computational complexity}

We analyze the computational cost of \Cref{alg:sparse-greedy} below.
\begin{description}
	\item[First ``for'' loop (Lines~\ref{alg:s-greedy-loop-kernel} to \ref{alg:s-greedy-loop-kernel-end}):] This loop runs for $h-1$ iterations.
	Each iteration computes	the subspace $\mathcal{K}_k$ at \Cref{alg:s-greedy-kernel} and runs the subroutine \texttt{greedy\_k} in \Cref{alg:greedy-k} at \Cref{alg:s-greedy-kernel-greedy}.
	\begin{description}
		\item[\boldmath Computing $\mathcal{K}_k$:] This step computes powers of $\bs{A}$ and then their null spaces.
		Both operations have complexity $O(n^3)$.
		Note that both the powers of $\bs{A}$ and their kernels can be computed just once	before running the algorithm,
		so that execution of \Cref{alg:s-greedy-kernel} requires only selecting the subspaces of interest.
		\item[Running \texttt{greedy\_k}:] This subroutine runs for at most $s$ iterations.
		Each iteration computes the cost for at most $m$ channels in \texttt{cand},
		selects the best channel,
		and runs the test at \Cref{alg:greedy-k-if} for each remaining candidate.
		The latter operation involves two projections onto $\mathcal{K}_k$ and the projection of either of the two projected subspaces onto the other,
		with complexity $O(n^3)$.
		Letting the computation of the cost function have complexity $O(n^\beta)$,
		the total complexity of \texttt{greedy\_k} is thus $O(s(mn^\beta + mn^3)) \subseteq O(sm\max\{n^\beta,n^3\})$.
	\end{description}
	The complexity of this loop is $O(hms\max\{n^\beta, n^3\})) $.
	\item[Second ``for'' loop (\Cref{alg:s-greedy-increase-rank,alg:s-greedy-increase-rank-end}):] This loop runs for at most $n$ iterations,
	with each iteration running the subroutine \texttt{greedy} in \Cref{alg:greedy} with rank check.
	This subroutine computes the best candidate channel at \Cref{alg:greedy-selection} and performs the check at \Cref{alg:greedy-if} for remaining candidates.
	\begin{description}
		\item[Selecting the best channel:] This operation computes the cost function at most $hm$ times,
		with total complexity $O(hmn^\beta)$.
		\item[Checking the rank:] This check computes the rank of an $n \times p$ matrix, with $p\le n$, at most $hm$ times, with total complexity $O(hmn^{3})$.
	\end{description}
	The total complexity of this loops is thus $O(n(hmn^\beta + hmn^{3})) \subseteq O(hm\max\{n^{\beta + 1},n^{4}\})$.
	\item[``While'' loop (Lines~\ref{alg:s-greedy-drop-cost} to~\ref{alg:s-greedy-drop-cost-end})] This loop runs the subroutine \texttt{greedy} at most $sh - n$ times without rank check,
	with total computational complexity in $O((sh - n)hmn^\beta)$.
\end{description}

Summing all three contributions and considering $sh \in O(n)$,
the $s$-sparse improved greedy algorithm has computational complexity of polynomial order $O(hm\max\{n^{\beta + 1},n^{4}\})$,
dominated by the second ``for'' loop.
Notably,
this is the same computational complexity of the naive greedy algorithm under point-wise sparsity constraint,
which coincides with running the second ``for'' loop without rank checking.}

\subsection{$s$-Sparse Markov Chain Monte Carlo}\label{sec:design-mcmc}

\begin{algorithm}[t]
	\caption{$s$-sparse MCMC}
	\label{alg:mcmc}
	\KwIn{Matrices $\bs{A},\bs{B}$, sparsity $s$, horizon $h$, cost $\rho$, parameters $\epsilon\ge0$, $\tinit>0$, $\tmin>0$, $\alpha\in(0,1)$, \texttt{it}, initial schedule $\sch[0]$.}
	\KwOut{Schedule $\mathcal{S}$.}
	$\sch \leftarrow \sch[0]$\;
	$\costmin\leftarrow \rho(\gram{\sch}{h} + \epsilon \bs{I})$\;
	$\texttt{rk}_W \leftarrow \rankil{\gram{\sch}{h}}$ \tcp*{optional, to check controllability}
	$T\leftarrow\tinit$\;
	\While{$T > \tmin$\label{alg:mcmc-tmin}}{
		\For{$\texttt{i} = 1,\dots,\texttt{it}$\label{alg:mcmc-it}}{
			$\sch'\leftarrow\sch$\;
			$k, c_k \sim U(\sch)$ \tcp*{sample from current schedule}\label{alg:mcmc-sample}
			$c'_k \sim U([m]\setminus\sch[k])$ \tcp*{sample candidate channel}\label{alg:mcmc-sample-cand}
			$\sch[k]' \leftarrow (\sch[k] \setminus \{c_k\}) \cup \{c'_k\}$\;
			\If{check rank\label{alg:mcmc-rank}}{
				\While{$\rankil{\gram{\sch'}{h}} < \texttt{rk}_W$}{
					\textbf{go to} Line~\ref{alg:mcmc-sample} or \textbf{continue}\;
				}
			}
			$\costcurr\leftarrow\rho(\gram{\sch'}{h} + \epsilon \bs{I})$\;
			$p\leftarrow \e^{-\frac{1}{T}(\costcurr - \costmin)}$\;
			\If(\tcp*[f]{w.p. $\min\{1,p\}$}){$p> n \sim U([0,1])$\label{alg:mcmc-replace}}{
				$\sch[k] \leftarrow \sch[k]'$\;
				$\costmin\leftarrow \costcurr$\;
				\If{check rank}{
					$\texttt{rk}_W \leftarrow \rankil{\gram{\sch'}{h}}$\;
				}
			}
		}
		$T \leftarrow \alpha T$\label{alg:mcmc-decrease-T}\;
	}
	\textbf{return} $\mathcal{S}$.
\end{algorithm}

MCMC is a randomized algorithm that approximately solves combinatorial problems \review{by sampling a Markov chain supported on the domain of the optimization variable}. We summarize its workflow in \cref{alg:mcmc}.
Every sample $\sch'$ replaces the current solution $\sch$ with a probability that decreases exponentially with the cost gap (\cref{alg:mcmc-replace}).
\review{MCMC is asymptotically optimal with infinite samples,
	\ie as the parameter $T$ goes to zero and the iterations \texttt{it} go to infinity.}
In practice,
$T$ is progressively reduced down to $\tmin > 0$ (\Cref{alg:mcmc-tmin}) and finite samples are drawn for each value of $T$ (\Cref{alg:mcmc-it}),
\review{causing approximation errors and a suboptimality gap that decreases with \texttt{it} and $\tmin^{-1}$~\cite{joseph2023minimal}.}
The interested readers are referred to~\cite{Sammut2010} for details.
We adapt MCMC to \Cref{prob:design} using the output of \cref{alg:sparse-greedy} as warm start $\sch[0]$.
To enforce $s$-sparsity,
we restrict candidate samples $\sch'$ to differ from the current solution $\sch$ by only one channel in one time step $k$ (\cref{alg:mcmc-sample,alg:mcmc-sample-cand}).

\review{Although MCMC has polynomial complexity with respect to its parameters $\alpha,\tinit,\tmin$, and \texttt{it},}
it generally converges slow in practice and requires a huge number of samples to achieve good solutions for large systems.
Moreover,
its randomized nature prevents formal controllability guarantees and the output schedule may violate constraint~\eqref{eq:prob-design-constr-controllability}.
This issue can be overcome by imposing that candidate samples $\sch'$ do not decrease the rank of the controllability Gramian (\cref{alg:mcmc-rank}),
which however may significantly slow down the algorithm.
\extended{}{Moreover,
if the initial schedule $\sch[0]$ does not yield controllability (\eg it is output from the naive greedy algorithm or it is random),
it is not possible to formally ensure that the output schedule $\sch$ makes the system controllable.}

\section{Numerical Experiments}\label{sec:experiments}

We test our algorithms
on two problem instances.\extended{\footnote{
	The code used for experiments is available at \url{https://github.com/lucaballotta/sparse-control}.
}}{}

\subsection{Revisited \cref{ex}}\label{sec:experiment-small}

\begin{table}
	\caption{Schedules and costs for \cref{ex}.}
	\label{table:experiment-small}
	\begin{center}
		\small
		\begin{tabular}{l|ccccc|c}
			\toprule
			&	$\sch[0]$	&	$\sch[1]$	&	$\sch[2]$	&	$\sch[3]$	&	$\sch[4]$	& $\tr{(\gram{\sch}{h})^{-1}}$\\
			\midrule
			fully actuated	&	$[7]$	&	$[7]$	&	$[7]$	&	$[7]$	&	$[7]$	& $1.7$\\
			naive greedy	& $1$	&	$4$	&	$1$	&	$2$	&	$7$	&	uncontrollable\\
			$s$-sparse greedy	& $1$	&	$4$	&	$4$	&	$4$	&	$4$	&	$5.0$\\
			\bottomrule
		\end{tabular}
	\end{center}
\end{table}

We use \cref{alg:sparse-greedy} to solve the design of \Cref{ex}.
The output schedule and cost are reported in \autoref{table:experiment-small},
together with the schedule output by the naive greedy algorithm and the cost of the fully actuated system for the sake of comparison.
Our proposed $s$-sparse greedy algorithm picks the key input channel $4$ that guarantees controllability according to the necessary condition~\eqref{eq:rank_condition} for all time steps $k=1,2,3,4$,
thanks to the smart pre-selection at \cref{alg:s-greedy-kernel,alg:s-greedy-kernel-greedy} of \cref{alg:sparse-greedy}.

\subsection{Experiment with Large-Scale Network}\label{sec:experiment-large}

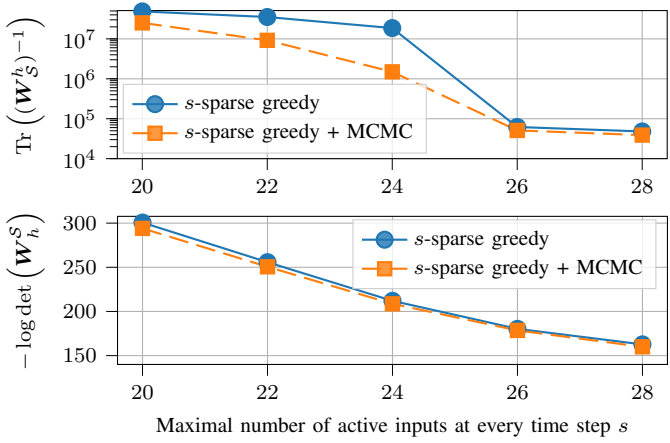
\begin{figure}
	\centering

\begin{tikzpicture}
	
	\definecolor{darkgray176}{RGB}{176,176,176}
	\definecolor{darkorange25512714}{RGB}{255,127,14}
	\definecolor{lightgray204}{RGB}{204,204,204}
	\definecolor{steelblue31119180}{RGB}{31,119,180}
	
	\begin{axis}[
		width=\linewidth,
		height=.4\linewidth,
		legend cell align={left},
		legend style={at={(.01,.27)},anchor=west,fill opacity=0.8, draw opacity=1, text opacity=1, draw=lightgray204,font=\footnotesize},
		ymode=log,
		tick align=outside,
		tick pos=left,
		x grid style={darkgray176},
		label style={font=\footnotesize},
		ticklabel style={font=\footnotesize},
		xmajorgrids,
		xmin=19.6, xmax=28.4,
		xtick style={color=black},
		y grid style={darkgray176},
		ylabel={$\tr{(\gram{\sch}{h})^{-1}}$},
		yticklabel style={
			/pgf/number format/fixed,
			/pgf/number format/precision=0
		},
		scaled y ticks=false,
		ymajorgrids,
		ymin=9999,
		ymax=51183786.4809276,
		ytick style={color=black}
		]
		\addplot [thick, steelblue31119180, mark=*, mark size=3, mark options={solid}]
		table {%
			20 48748317.2804988
			22 35287711.0878443
			24 18636185.3486982
			26 62134.5361337786
			28 47992.3111678588
		};
		\addlegendentry{$s$-sparse greedy}
		\addplot [thick, darkorange25512714, dash pattern=on 7.4pt off 3.2pt, mark=square*, mark size=2.5, mark options={solid}]
		table {%
			20 25161692.8907769
			22 9203105.71242875
			24 1496161.62057947
			26 50922.7236728156
			28 38933.2719226505
		};
		\addlegendentry{$s$-sparse greedy + MCMC}
	\end{axis}
	
\end{tikzpicture}\\

\begin{tikzpicture}
	
	\definecolor{darkgray176}{RGB}{176,176,176}
	\definecolor{darkorange25512714}{RGB}{255,127,14}
	\definecolor{lightgray204}{RGB}{204,204,204}
	\definecolor{steelblue31119180}{RGB}{31,119,180}
	
	\begin{axis}[
		width=\linewidth,
		height=.4\linewidth,
		legend cell align={left},
		legend style={fill opacity=0.8, draw opacity=1, text opacity=1, draw=lightgray204, font=\footnotesize},
		tick align=outside,
		tick pos=left,
		x grid style={darkgray176},
		label style={font=\footnotesize},
		ticklabel style={font=\footnotesize},
		xlabel={Maximal number of active inputs at every time step \(\displaystyle s\)},
		xmajorgrids,
		xmin=19.6, xmax=28.4,
		xtick style={color=black},
		y grid style={darkgray176},
		ylabel={$-\log\det\lb\gram{h}{\sch}\rb$},
		ymajorgrids,
		ymin=140, ymax=307.671486055618,
		ytick style={color=black}
		]
		\addplot [thick, steelblue31119180, mark=*, mark size=3, mark options={solid}]
		table {%
			20 300.636679184292
			22 255.767238397322
			24 211.998414544028
			26 180.297030870885
			28 162.689412948341
		};
		\addlegendentry{$s$-sparse greedy}
		\addplot [thick, darkorange25512714, dash pattern=on 7.4pt off 3.2pt, mark=square*, mark size=2.5, mark options={solid}]
		table {%
			20 294.05759665048
			22 250.554042010149
			24 208.648932562124
			26 178.521446353313
			28 159.940541757783
		};
		\addlegendentry{$s$-sparse greedy + MCMC}
	\end{axis}
	
\end{tikzpicture}
	\vspace{-3mm}
	\caption{Cost for experiment on large-scale network.
		The parameters of $s$-sparse MCMC are $\tinit=1$, $\tmin=10^{-7}$, $\alpha=0.1$, $\texttt{it}=5000$.
	}
	\label{fig:experiment-large}
\end{figure}

We benchmark \Cref{alg:sparse-greedy} on a large-scale system with $n=m=50$,
where $\bs{A}$ is the adjacency matrix of a random geometric graph scaled by the number of nodes and $\bs{B}$ is the identity matrix.
This setup mimics a network system where node (subsystem) $i$ can be directly controlled through the $i$th column of $\bs{B}$.
We generate sparse random geometric graphs with nodes in the unit square $[0,1]^2$ and radius $0.1$,
and choose time horizon $h=n=50$.
The nullity of $\bs{A}$ ranges from $10$ to $20$ and $\keril{\bs{A}} = \keril{\bs{A}^k} \; \forall k\ge1$,
meaning that $R=1$ and the critical time step for controllability is only $k = h-1$ based on \cref{lem:suff}.
The naive greedy algorithm always yields uncontrollable systems when choosing the minimal admissible sparsity $s = n-\rankil{\bs{A}}$,
even when refining its output with $s$-sparse MCMC.
In contrast,
our $s$-sparse greedy algorithm successfully yields controllable systems on all tried instances.

Finally,
we compare the costs achieved with the $s$-sparse greedy selection (\cref{alg:sparse-greedy}) and $s$-sparse MCMC (\cref{alg:mcmc}) as the sparsity constraint $s$ varies.
Figure~\ref{fig:experiment-large} illustrates the results for a system with $n-\rankil{\bs{A}}=20$ where $s$ increases from $20$ (the minimum for $s$-sparse controllability) to $28$.
Both curves are decreasing,
meaning that selecting more inputs reduces the control effort.
Also,
while MCMC yields smaller costs thanks to its exploratory approach,
the costs obtained with the $s$-sparse greedy are only slightly larger in most cases,
especially with $\rho(\cdot) = -\log\det(\cdot)$,
validating its effectiveness.
	\review{

\section{Conclusion}\label{sec:conclusion}
We investigated the sparse actuator scheduling problem for linear systems to ensure controllability with minimal control effort. 
We characterized feasible actuator schedules,  
designed a greedy algorithm with formal controllability guarantee, and enhanced it using MCMC-based randomized optimization. 
Numerical tests prove our algorithms effective even when the naive greedy selection fails. 
Future work could explore stricter sparsity constraints, such as slowly varying active input sets. }
	
	\appendices
	\crefalias{section}{appendix}

\section{Proof of \Cref{lem:suff}}\label{app:suff}

We define $\bs{K}^{(k)}=\bs{I}-\bs{A}^k(\bs{A}^k)^{\dagger}$ as the projection matrix onto $\keril{(\bs{A}^{k})^\top}$, where $(\cdot)^{\dagger}$ is the pseudoinverse. 
As the subspace orthogonal to $\keril{(\bs{A}^{k})^\top}$ is $\imil{\bs{A}^{k}}$, we get
\begin{equation}\label{eq:orthogonalspaces}
	\ker{\lb\bs{A}^{k}\rb^\top} \ominus \ker{\lb\bs{A}^{k-1}\rb^\top} = \im{\bs{K}^{(k)}\bs{A}^{k-1}}.
\end{equation}
We note that~\eqref{eq:suff_feasibility} trivially holds if $R=0$ by setting $k=R+1$.
Then,
to prove~\eqref{eq:suff_feasibility} for all $k\in[R]$ it suffices to prove that there exists an index set $\cl{S}_{h-k}$ for each $k\in[R]$ such that
\begin{equation}\label{eq:suff_condition}
	\im{\bs{K}^{(k)}\bs{A}^{k-1}} = \im{\bs{K}^{(k)}\bs{A}^{k-1}\bs{B}_{\cl{S}_{h-k}}}. 
\end{equation}
Since the system is controllable,
it holds
\begin{equation} \label{eq:span_AB}
	\begin{aligned}
		\im{\bs{K}^{(k)}\bs{A}^{k-1}} &=\im{\bs{K}^{(k)}\bs{A}^{k-1}\bs{\Phi}^n} \\
		&=\im{\bs{K}^{(k)}\bs{A}^{k-1}\bs{B}},
	\end{aligned}
\end{equation}
because $\bs{K}^{(k)}\bs{A}^{k-1}\bs{A}^i=[\bs{I}-\bs{A}^{k}(\bs{A}^{k})^{\dagger}]\bs{A}^{k}\bs{A}^{i-1}=\bs{0}$, for $i>0$. 
Applying the Sylvester rank inequality to~\eqref{eq:orthogonalspaces} yields 
\begin{multline}
	\rank{\bs{K}^{(k)}\bs{A}^{k-1}} =\ls n-\rank{\bs{A}^{k}} \rs-\ls n-\rank{\bs{A}^{k-1}}\rs\\
	=\rank{\bs{A}^{k-1}} - \rank{\bs{A}^{k}}    \leq n-\rank{\bs{A}}\leq s.
\end{multline}
From \eqref{eq:span_AB}, we conclude that there exist $s$ columns in $\bs{K}^{(k)}\bs{A}^{k-1}\bs{B}$,
indexed by $\cl{S}_{h-k}$,
such that \eqref{eq:suff_condition} holds.
Also, for any $\bs{z}\in\keril{(\bs{A}^R)^{\top}}$, we have $\bs{z} = [\bs{I}-\bs{K}^{(1)}]\bs{z}+\bs{K}^{(1)}\bs{z}$. Here, $\bs{K}^{(1)}\bs{z}=\bs{K}^{(1)}\bs{B}_{\sch[h-1]}\bs{v}(h-1)$
for some $\bs{v}(h-1)$ due to \eqref{eq:suff_condition} with $k=1$. 
Hence, defining $\bar{\bs{z}}(1) \doteq  [\bs{I}-\bs{K}^{(1)}][\bs{z}-\bs{B}_{\sch[h-1]}\bs{v}(h-1)]$, we have
\begin{equation}
	\bs{z} = \bar{\bs{z}}(1)+\bs{B}_{\sch[h-1]}\bs{v}(h-1).
\end{equation}
Recursively using \eqref{eq:suff_condition}, there exist $\{\bs{v}(k)\}_{k=1}^R$ such that
\begin{equation}\label{eq:z}
	\begin{aligned}
		\bs{z} &= \ls \bs{I}-\bs{K}^{(2)}\rs\bar{\bs{z}}(1)+\bs{K}^{(2)}\bar{\bs{z}}(1)+\bs{B}_{\sch[h-1]}\!\bs{v}(h-1)\\
		&= \bar{\bs{z}}(2)+\sum_{k=1}^{2}\bs{A}^{k-1}\bs{B}_{\sch[h-k]}\bs{v}(h-2)\\
		&=\bar{\bs{z}}(R)+\sum_{k=1}^{R}\bs{A}^{k-1}\bs{B}_{\sch[h-k]}\bs{v}(h-k),
	\end{aligned}
\end{equation}
where $\bar{\bs{z}}(k)=[\bs{I}-\bs{K}^{(k)}][\bar{\bs{z}}(k-1)-\bs{B}_{\sch[h-k]}\bs{v}(h-k)]$.
Since $\bs{z}\in\keril{(\bs{A}^R)^{\top}}$, multiplying~\eqref{eq:z} with $\bs{K}^{(R)}$ gives 
\begin{equation}
	\bs{K}^{(R)}\bs{z}=\bs{z} = \bs{K}^{(R)}\sum_{k=1}^{R}\bs{A}^{k-1}\bs{B}_{\sch[h-k-1]}\bs{v}(k),
\end{equation}
proving the desired result \eqref{eq:span_null}.

\section{Proof of \Cref{thm:kernel-condition}}\label{app:kernel-condition}

We begin by noting that $\bb{R}^n=\keril{(\bs{A}^R)^{\top}}\oplus\cl{A}_R$, 
where $\cl{A}_R \doteq \imil{\bs{A}^R}$. 
From \Cref{lem:suff}, it suffices to prove that there exists an $s$-sparse schedule $\sch\in\cl{T}_m^h$ such that 
\begin{equation}\label{eq:span_range}
	\cl{A}_R = \im{\begin{bmatrix}
			\bs{A}^{h-1}\bs{B}_{\sch[0]}& \bs{A}^{h-2}\bs{B}_{\sch[1]} & \ldots & \bs{A}^{R}\bs{B}_{\sch[h-R-1]}
	\end{bmatrix}}.
\end{equation}
To this end, let the real Jordan canonical of $\bs{A}$ be
\begin{equation}\label{eq:Jordan1}
	\bs{A} =\bs{P}^{-1}\begin{bmatrix}
		\bs{J} & \bs{0}\\
		\bs{0} & \bs{N}
	\end{bmatrix}\bs{P}=\begin{bmatrix}
		\bs{P}^{(1)} \\ \bs{P}^{(2)}
	\end{bmatrix}^{-1}\begin{bmatrix}
		\bs{J} & \bs{0}\\
		\bs{0} & \bs{N}
	\end{bmatrix}\begin{bmatrix}
		\bs{P}^{(1)} \\ \bs{P}^{(2)}
	\end{bmatrix},
\end{equation}
where $\bs{P}= [\bs{P}^{(1)\T} \ \bs{P}^{(2)\T}]^{\T}$ is an invertible matrix, \review{with the columns of $\bs{P}^{(1)}$ and $\bs{P}^{(2)}$ are the generalized eigenvectors of $\bs{A}$ corresponding to the nonzero and zero eigenvalues of $\bs{A}$, respectively. Also,} the square matrices $\bs{J}\in\bb{R}^{J\times J}$ and $\bs{N}$ are formed by the Jordan blocks of $\bs{A}$ corresponding to its nonzero and zero eigenvalues, respectively. 
We see that $\bs{N}^R=\bs{0}$, for any $k\geq R$. As a result, $J\doteq\rank{\bs{J}}\leq n-R$, and we deduce 
\begin{equation}\label{eq:JordanA_v2}
	\cl{A}_R = \im{\bs{P}^{-1}\begin{bmatrix}
			\bs{J}^R & \bs{0}\\
			\bs{0} &\bs{0}
		\end{bmatrix}\bs{P}} = \im{\begin{bmatrix}
			\bs{J}^R \\
			\bs{0}
	\end{bmatrix}}.
\end{equation}
Further, premultiplying \eqref{eq:systemmodel} with $\bs{P}^{(1)}$ gives
\begin{equation}
	\bs{P}^{(1)}\bs{x}(k+1)=\bs{J}\bs{P}^{(1)} \bs{x}(k)+\bs{P}^{(1)}\bs{B}\bs{u}(k).
\end{equation}
The linear dynamical system $(\bs{J},\bar{\bs{B}})$ with $\bar{\bs{B}}\doteq\bs{P}^{(1)}\bs{B}$ is $s$-sparse controllable for any $s\geq 1$,
because $(\bs{A},\bs{B})$ is controllable and $\bs{J}$ is invertible. Hence, by \Cref{thm:sparse-schedule}, there exist sets $\sch[0],\sch[1],\ldots,\sch[J-1]$ such that $\lv\sch[k]\rv\leq s$ and 
\begin{equation}
	\bb{R}^{J}= \im{\begin{bmatrix}
			\bs{J}^{J-1}\bar{\bs{B}}_{\sch[0]} & \bs{J}^{J-2}\bar{\bs{B}}_{\sch[1]}&\ldots & \bar{\bs{B}}_{\sch[J-1]}
	\end{bmatrix}}.
\end{equation}
Consequently, from \eqref{eq:JordanA_v2}, we have
\begin{equation}
	\begin{aligned}
		\cl{A}_R&= \im{\begin{bmatrix}
				\bs{J}^{R+J-1}\bar{\bs{B}}_{\sch[0]} & \bs{J}^{R+J-2}\bar{\bs{B}}_{\sch[1]}& \!\!\! \dots \!\!\!  &\bs{J}^{R}\bar{\bs{B}}_{\sch[J-1]}\\
				\bs{0} &\bs{0} & \!\!\! \dots \!\!\! & \bs{0}
		\end{bmatrix}}\\
		&= \im{\begin{bmatrix}
				\bs{A}^{R+J-1}\bs{B}_{\sch[0]} & \bs{A}^{R+J-2}\bs{B}_{\sch[1]} & \!\!\!\dots\!\! & \bs{A}^R\bs{B}_{\sch[J-1]}
		\end{bmatrix}}.
	\end{aligned}
\end{equation}
Due to $h\geq n$, it follows $h-R\geq n-R\geq J$ and thus $R+J\le h$, implying \eqref{eq:span_range} holds, and the proof is complete.
	
	\extended{
		\bibliographystyle{IEEEtran}
		\bibliography{biboptions,bibfile}
		}{

	}
	
\end{document}